\documentclass[english
              ,twoside
              ]{amsart}
\usepackage[T1]{fontenc}
\usepackage{lmodern}
\usepackage[utf8]{inputenc}
\usepackage{babel}
\usepackage{hyperref}
\usepackage{cas-math}
\usepackage{cas-paper}
\usepackage[arrow,cmtip,matrix]{xy}

\newcommand{\urlprefix}{}

\newcommand{\subclass}[1]{}

\DeclareMathOperator{\End}{End}

\newcommand{\sSG}{\overline{SG}}

\begin{document}

\title{Paths of homomorphisms from stable Kneser graphs}

\author{Carsten Schultz}
\address{Institut für Mathematik, MA 6-2\\
Technische Universität Berlin\\
D-10623 Berlin, \hbox{Germany}}
\email{carsten@codimi.de}
\date{2nd June 2010}

\begin{abstract}
We denote by $SG_{n,k}$ the stable Kneser graph (Schrijver graph) of
stable $n$-subsets of a set of cardinality~$2n+k$.  For
$k\equiv3\pmod4$ and $n\ge2$ we show that there is a component of the
$\chi$-colouring graph of~$SG_{n,k}$ which is invariant under the
action of the automorphism group of~$SG_{n,k}$.  We derive that there
is a graph $G$ with $\chi(G)=\chi(SG_{n,k})$ such that the complex
$\Hom(SG_{n,k}, G)$ is non-empty and connected. In particular,
 for $k\equiv3\pmod4$ and $n\ge2$ the graph $SG_{n,k}$ is not a test graph.
\end{abstract}

\maketitle

\section{Introduction}
For graphs $G$ and~$H$, the complex $\Hom(G, H)$ is a cell complex
whose vertices are the graph homomorphisms from $G$ to~$H$ and whose
topology captures global properties of the set of these homomorphisms.
Research on these complexes in recent years has been driven by the
concept of a test graph.  In this work we present a result in this
area which can be formulated naturally in the category of graphs
without mentioning complexes.
\begin{def*}[\cite{babson-kozlov-i}]
A graph~$T$ is a \emph{test graph} if for all graphs~$G$ and $r\ge0$
such that the cell complex $\Hom(T, G)$ is $(r-1)$-connected we have
$\chi(G)\ge r+\chi(T)$.
\end{def*}
A cell complex is said to be $0$-connected if it is non-empty.  If
there is a graph homomorphism from $T$ to~$G$, then the chromatic
number of~$G$ is at least as large as that of~$T$.
Therefore the condition for~$r=0$ is always satisfied.

A cell complex is said to be $1$-connected if it is non-empty and
path-connected.  If $f,g\colon T\to G$ are graph homomorphism, then by
definition there is an edge from $f$ to~$g$ in $\Hom(T, G)$ if and
only if $f$ and~$g$ differ at exactly one vertex of~$T$.  This is all
that is needed to understand the condition for~$r=1$, and since that
is all that we will be interested in in this work, we omit the
description of the higher dimensional cells of~$\Hom(T, G)$.

The seminal result regarding test graphs is that $K_2$~is a test
graph.  This is a translation of a result by Lov\'asz~\cite{lovasz}
which predates the definiton of the complex~$\Hom(T, G)$.  While some
effort has gone into proving that certain graphs are test
graphs~\cite{babson-kozlov-i,babson-kozlov-ii,hom-loop,anton-cas,sg},
it was not known from the beginning if possibly every graph is a test
graph.  An example of a graph which is not a test graph was given by
Hoory and Linial
.
\begin{thm*}[\cite{linial-test-graph}]
There is a graph $T$ such that $\Hom(T, K_{\chi(T)})$ is connected.
\end{thm*}
The graph~$T$ in this example fails to be a test graph
already for~$r=1$.  Since $\Hom(T, K_{\chi(T)})$ is non-empty by
definition, we would otherwise have~$\chi(K_{\chi(T)})\ge \chi(T)+1$,
which is absurd.

So this graph~$T$ fails to be a test graph in the most fundamental
way.  Furthermore, the fact that $\Hom(T, K_{\chi(T)})$ is connected also
implies
that $\Hom(K_2, T)$ cannot be $(\chi(T)-3)$-connected, i.e.~the
chromatic number of~$T$ is not detected by the test graph~$K_2$. 
The proof of this uses some easy topology and functorial properties of~$\Hom$.
One might therefore ask if there are graphs~$T$ whose chromatic numbers
are detected by $K_2$ and which still fail to be test graphs
for~$r=1$.  We will find such examples among the stable Kneser graphs.

The stable Kneser graphs, first introduced by Schrijver,
form a two parameter
family of graphs~$SG_{n,k}$,
see \prettyref{def:sg}.  We list
some facts that are known of them.
\begin{itemize}
\item
The chromatic number of $SG_{n,k}$ equals~$k+2$
(Schrijver~\cite{schrijver}, extending work of Lovasz~\cite{lovasz}
and B{\'a}r{\'a}ny~\cite{barany78}).
\item Indeed, $\Hom(K_2, SG_{n,k})$ is homotopy equivalent to a $k$-sphere
as shown by Björner and de~Longueville~\cite{anders-mark}
(a simplified proof can be found in~\cite{sg}).
\item
The graph~$SG_{n,k}$ is vertex critical, i.e.~every induced subgraph
on a proper subset of its set of vertices is
$(k+1)$-colourable~\cite{schrijver}.
\item The stable Kneser graph~$SG_{1,k}$ is a complete graph and hence
  a test graph by work of Babson and Kozlov~\cite{babson-kozlov-i}.
  The stable Kneser graph $SG_{n,1}$ is a cycle on $2n+1$
  vertices and hence also a test graph by their
  work~\cite{babson-kozlov-ii}.
\item
For $n\ge2$ and $k\ge1$, Braun~\cite{braun-sg} has shown the
automorphism group of~$SG_{n,k}$ to be the symmetry group of a
$(2n+k)$-gon.
\end{itemize}
Since $\Hom(K_2, SG_{n,k})$ is topologically as nice as one might
hope, one might have thought that all stable Kneser graphs are test
graphs.  It turns out, however, that very few of them are.
\begin{thm*}[{\cite[10.5--10.10]{sg}}]
If $k\notin\set{0,1,2,4,8}$ then there is an $N(k)$ such that for all
$n\ge N(k)$ the graph~$SG_{n,k}$ is not a test graph. For
$k\equiv3\pmod4$ the graph~$SG_{n,k}$ fails to be a test graph for~$r=1$.
\end{thm*}
The proof of this theorem studies the action of the automorphism group
of~$SG_{n,k}$ on the space~$\Hom(K_2, SG_{n,k})$ using methods from
algebraic topology.  The current work gives an elementary proof for
the case~$k\equiv3\pmod4$, which yields the stronger result that in
these cases we can actually set~$N(k)=2$.  Its main result is thus the
following.
\begin{thm*}
Let $k\equiv3\pmod4$ and $n\ge2$.  There is a graph $G$ with
$\chi(G)=\chi(SG_{n,k})=k+2$ and such that $\Hom(SG_{n,k}, G)$ is
non-empty and connected.
\end{thm*}
In the proof, which is the combination of the following two theorems,
the automorphism group of~$SG_{n,k}$ again plays an important role.
However, we only have to study its action on the set of path
components of $\Hom(SG_{n,k}, K_{k+2})$.  While we know from the above
discussion that the complex $\Hom(SG_{n,k}, K_{k+2})$ cannot be
connected, we will show in \prettyref{sec:sg} the following weaker result.
\begin{thm*}[\ref{thm:sg}]
Let $k\equiv3\pmod4$ and $n\ge2$.
Then there is a component of $\Hom(SG_{n,k}, K_{k+2})$ which is
invariant under the action of the automorphism group of
$SG_{n,k}$.
\end{thm*}
This of course relies on Braun's result on the structure of
$\Aut(SG_{n,k})$.  In \prettyref{sec:crit} we give a self-contained
proof the following general criterion from~\cite{sg}.
\begin{thm*}[\ref{thm:crit}]
Let $T$ be a finite, vertex critical graph. If there is a
component of $\Hom(T, K_{\chi(T)})$ which is $\Aut(T)$-invariant, then
there exists a graph~$G$ such that $\Hom(T, G)$ is non-empty and
connected and $\chi(G)=\chi(T)$.
\end{thm*}
That every endomorphism of $SG_{n,k}$ is an automorphism follows
immediately from vertex criticality.  In this sense, our proof also
relies on Schrijver's result that $SG_{n,k}$ is vertex crtitical.

\section{Constructions in the category of graphs}
We recall some definitions related to the category of graphs.  Details
can be found in~\cite{anton-graph-cat}.

A graph $G$ consists of a \emph{vertex set} $V(G)$ and a symmetric
binary relation $E(G)\subset V(G)\times V(G)$.  The relation is called
\emph{adjacency}, adjacent vertices are also called \emph{neighbours}
and elements of $E(G)$ \emph{edges}.  We also write $u\nbo v$ for
$(u,v)\in E(G)$.  We point out that we allow loops, i.e.~edges of the
form $(v,v)$.  A graph with a loop at every vertex is 
called reflexive.

A \emph{graph homomorphism} $f\colon G\to H$ is a function $f\colon
V(G)\to V(H)$ between the vertex sets which preserves the adjacency
relation, $(f(u),f(v))\in E(H)$ for all $(u,v)\in E(G)$.  When
discussing the structure of the set of graph homomorphisms from $G$
to~$H$, it will be useful to not only consider the cell complex
$\Hom(G, H)$, but also the closely related graph~$[G, H]$.  This
graph, sometimes also written~$H^G$, is defined by
\begin{align*}
V([G,H])&=V(H)^{V(G)},\\
E([G,H])&=\set{(f,g)\colon 
\text{$(f(u),g(v))\in E(H)$ for all $(u,v)\in E(G)$}}.
\end{align*}
In particular, we have $f\nbo f$ if and only if $f$~is a graph homomorphism.
Furthermore, we have the following.
\begin{lem}
Let $f,g\colon G\to H$ be graph homomorphisms and $G$~loopless.  Then
$f\nbo g$ in~$[G,H]$ if and only if each $h\colon V(G)\to V(H)$ with
$h(u)\in\set{f(u),g(u)}$ for all $u\in V(G)$ is a graph homomorphism.
\qed
\end{lem}
It follows that $f\nbo g$ if an edge joins the two graph homomorphism
$f$ and~$g$ in $\Hom(G, H)$.  If on the other hand $f\nbo g$, then we
can get from $f$ to~$g$ in $\Hom(G, H)$ by changing the values at the
vertices of~$G$ in any order.  Therefore, for questions of
connectivity it does not matter whether we work in $\Hom(G, H)$ or in
the induced subgraph of looped vertices of $[G, H]$.  That two graph
homomorphisms are in the same component can now be reformulated as
follows.
\begin{defn}
For $n\ge0$ we define a reflexive graph $I_n$ by
$V(I_n)=\set{0,\dots,n}$, $i\nbo j\iff \abs{i-j}\le1$.  For graph
homomorphisms $f,g\colon G\to H$ we write $f\homot g$ if and only if
there is an $n\ge0$ and a graph homomorphism $p\colon I_n\to[G,H]$
with $p(0)=f$, $p(n)=g$.
\end{defn}

The construction $[\bullet,\bullet]$ is an \emph{inner hom},
intimately related to products.  The product of two graphs in the
category~$\mathcal G$ of graphs and graph homomorphisms is given by
\begin{align*}
V(G\times H)&=V(G)\times V(H),\\
E(G\times H)&=
\set{((u,u'),(v,v'))\colon 
 \text{$(u,v)\in E(G)$, $(u',v')\in E(H)$}}.
\end{align*}
For every graph $G$, the functor $[G, \bullet]$ is a right adjoint to
$\bullet\times G$, i.e. there is a natural equivalence
\begin{equation}\label{eq:adj}
\mathcal{G}(Z\times G, H)\isom\mathcal{G}(Z, [G,H]).
\end{equation}
For example, that looped vertices of $[G,H]$ correspond to graph
homomorphisms can be derived as a formal consequence of this
adjunction.  If $\ug$ denotes the terminal object of~$\mathcal
G$, we obtain
\[\mathcal G(G, H)
\isom\mathcal G(\ug\times G, H)
\isom\mathcal G(\ug, [G, H]),
\]
and since $\ug$ is a graph consisting of one vertex and one loop, the
graph homomorphisms from $\ug$ to some other graph correspond to the
looped vertices of that graph.

\section{A criterion for not being a test graph}\label{sec:crit}
If $T$ is a vertex critical finite graph, then every graph
homomorphism $T\to T$ will have to be surjective and hence bijective.
In other words, every endomorphism of~$T$ is an automorphism.  For
graphs with this property we obtain the following criterion to decide
whether they satisfy the test graph property for~$r=1$.

\begin{thm}\label{thm:crit}
Let $T$ be a finite graph such that $\End(T)=\Aut(T)$.  If there is a
component of $\Hom(T, K_{\chi(T)})$ which is $\Aut(T)$-invariant, then
there exists a graph~$G$ with $\chi(G)=\chi(T)$ such that $\Hom(T, G)$
is non-empty and connected.
\end{thm}

\begin{rem}
This is the case $r=1$, $s=\chi(T)+1$ of the implication
$\text{(iv)}\Rightarrow\text{(ii)}$ of \cite[Thm~10.1]{sg}.  The proof
there uses results from \cite{DocUni} and \cite{anton-cas}.  Here we
are only interested in dimension~$1$, and for this easier case we can
give a self-contained proof which follows the same lines.
\end{rem}

\begin{rem}
The converse holds in general without conditions on~$T$. If $f\colon
T\to G$ and $c\colon G\to K_{\chi(T)}$, then $c\cmps f$ is a vertex of
$\Hom(T, K_{\chi(T)})$, and if $\Hom(T, G)$ is connected, then the
component of that vertex will be invariant under~$\Aut(T)$:  If
$\gamma\in\Aut(T)$ then $f\cmps\gamma\homot f$ implies $c\cmps
f\cmps\gamma\homot c\cmps f$.
\end{rem}

\begin{proof}[Proof of \prettyref{thm:crit}]
Let $c\colon T\to K_{\chi(T)}$ be a colouring which lies in the
invariant component.  Then for every $\gamma\in\Aut(T)$ there is an
$n\ge0$ and a graph homomorphism $I_n\to[T, K_{\chi(T)}]$ with
$0\mapsto c$ and $n\mapsto c\gamma$.  We can assemble these into a
graph homomorphism $X\to[T, K_{\chi(T)}]$, where $X$ is a reflexive
graph consisting of a vertex~$u$ and for every $\gamma\in\Aut(T)$ a
path from $u$ to a vertex~$v_\gamma$ and the graph homomorphism such
that $u\mapsto c$ and $v_\gamma\mapsto c\gamma$.  We will later
require that the paths have a certain minimal length, which we can
arrange.

The graph homomorphism which we have just constructed gives rise via
\prettyref{eq:adj} to a
graph homomorphism $f\colon X\times T\to K_{\chi(T)}$ with
$f(v_\gamma,t)=c(\gamma(t))=f(u,\gamma(t))$ for all $\gamma\in\Aut(T)$
and $t\in V(T)$.  We define an equivalence relation on $V(X\times T)$
such that $(v_\gamma, t)\eqrel (u,\gamma(t))$ and obtain a commutative
diagram
\[\xymatrix{
X\times T\ar[r]^f\ar@{->>}[d]^q
&K_{\chi(T)}\\
G\deq(X\times T)/\eqrel
\ar[ru]_{\bar f}
}\]
with $q$ the quotient map.  Now let $j\colon T\to G$ be the inclusion
at~$u$, more formally $j=q\cmps(\const_u,\id_T)$.  Now we already know
that $\Hom(T, G)\ne\emptyset$ and $\chi(G)=\chi(T)$.  We will proceed
to show that $\Hom(T, G)$ is connected.  

We first note that $\const_u\homot\const_{v_\gamma}\colon T\to X$ for
all~$\gamma\in\Aut(T)$, and hence 
\[j\homot q\cmps(\const_{v_\gamma},\id_T)
=q\cmps(\const_u,\gamma)=j\gamma.\] Now let $g\colon T\to G$ be an
arbitrary graph homomorphism.  Let $\eqrel'$ be the equivalence
relation on~$V(X)$ given by $v_\gamma\eqrel' u$.  The reflexive
graph~$X/\eqrel'$ is a bouquet of circles, and there is a natural surjection
$G\to X/\eqrel'$.  Since we may assume to haven chosen the graph~$X$
large enough (this depending on~$T$, not on~$g$), the image of the
composition $T\xto g G\to X/\eqrel'$ will miss at least one vertex of
each of the circles.  Let $X'$ be the graph obtained from $X/\eqrel'$
by removing one these vertices from each circle.  We denote the vertex
corresponding to~$u$ by~$u'$.  The preimage of $X'$ in~$G$ is
isomorphic to $X'\times T$.  Let $h\colon X'\times T\to G$ be the
corresponding embedding such that $j=h\cmps(\const_{u'},\id_T)$.  
This defines a commutative diagram
\[\xymatrix{
& X'\times T \ar[d]^h 
\\ 
T\ar[ur]^{\tilde g} \ar[r]^g & G.  
}\] 

Now $\tilde g=(\tilde g_1,\tilde g_2)$.  Since $X'$ is a reflexive
tree with
loops, we have $\tilde g_1\homot\const_{u'}$.  Since
$\End(T)=\Aut(T)$, there is a $\gamma\in\Aut(T)$ such that $\tilde
g_2=\gamma$.  Therefore 
\[g\homot h\cmps(\const_{u'},\gamma)=h\cmps(\const_{u'},\id_T)\cmps\gamma
=j\gamma\homot j.
\]
Since $g$ was arbitrary, this shows that $\Hom(T, G)$ is connected.
\end{proof}

\section{Paths of colourings of stable Kneser graphs}\label{sec:sg}

\begin{defn}\label{def:sg}
Let $n\ge1$, $k\ge0$ and $m=2n+k$.  The \emph{Kneser graph} $KG_{n,k}$
is the graph whose vertices are the $n$-element subsets of
$\set{0,\dots,m-1}$ and in which two of them are adjacent if and only
if they are disjoint.
We call a subset $S$ of
$\set{0,\dots,m-1}$ \emph{semi-stable}, if $\set{i,i+1}\not\subset S$
for all~$0\le i\le m-2$, and \emph{stable}, if additionally
$\set{0,m-1}\not\subset S$.  
The \emph{stable Kneser graph}  $SG_{n,k}$ 
is the induced subgraph of $KG_{n,k}$ on the set of stable sets.
The \emph{semi-stable Kneser graph} $\sSG_{n,k}$ 
is the induced subgraph of $KG_{n,k}$ on the set of semi-stable sets.
\end{defn}

\begin{defn}
We call the graph homomorphism
\begin{align*}
c_{n,k}\colon\sSG_{n,k}&\to K_{k+2},\\
S&\mapsto \min S
\end{align*}
the \emph{canonical colouring of $\sSG_{n,k}$} and its restriction the 
\emph{canonical colouring of $SG_{n,k}$}.
\end{defn}

We will prove the following.
\begin{thm}\label{thm:sg}
Let $k\equiv3\pmod4$ and $n\ge2$.  Then for any
$\gamma\in\Aut(SG_{n,k})$ and $c\colon SG_{n,k}\to K_{k+2}$ the
canonical colouring there is a path in $\Hom(SG_{n,k})$ from $c$ to
$c\gamma$.
\end{thm}

We consider automorphisms of~$K_{k+2}$ before turning to automorphisms
of~$SG_{n,k}$.

\begin{prop}\label{prop:sg}
Let $n\ge2$, $k\ge1$.  Let $\pi\in A_{k+2}$ be an even permutation of
the vertices of $K_{k+2}$.  Then 
$c_{n,k}\homot\pi\cmps c_{n,k}\colon\sSG_{n,k}\to K_{k+2}$.
\end{prop}

\begin{proof}
Let $m=2n+k$.
We will assume that $\pi$ is a cycle of the form 
$(i\; i{+}1\;i{+}2)$ with
$0\le i<k$.  This is possible, since these permutations
generate~$A_{k+2}$.  The proof will be by induction on~$k$.  We
distinguish three cases.

\underline{$i>0$.}  We note that $c(S)\ge i$ if and only if
$S\subset\set{i,\dots,m-1}$.  The induced subgraph on the set of these
vertices is isomorphic to $\sSG_{n,k-i}$ via $S\mapsto S-i$.
Therefore a path from $c_{n,k-i}$ to $(0\;1\;2)\cmps c_{m,k-i}$, which
exists by induction, can be extended to a path from $c_{n,k}$ to
$\pi\cmps c_{n,k}$ by fixing all colours less than~$i$.

\underline{$i=0$, $k>1$.}
We define
\begin{align*}
c'\colon \sSG_{n,k}&\to K_{k+2},\\
S&\mapsto
\begin{cases}
k+1,&m-1\in S,\\
\min S,&\text{otherwise}.
\end{cases}
\end{align*}
Obviously $c'$ is a graph homomorphism.  If $S\nbo S'$ then
$m-1\notin\emptyset=S\intersect S'$ and hence $c_{n,k}(S)=c'(S)$ or
$c_{n,k}(S')=c'(S')$.  This shows~$c'\nbo c_{n,k}$.
The induced subgraph on those $S$ for which
$c'(S)<k+1$ equals $\sSG_{n,k-1}$. The restriction of $c'$ to that
subgraph equals $c_{n,k-1}$.  Therefore a path from $c_{n,k-1}$ to
$\pi\cmps c_{n,k-1}$ extends to one from $c'$ to $\pi\cmps
c'$.  Hence $c_{n,k}\nbo c'\homot\pi\cmps c'\nbo\pi\cmps c_{n,k}$.

\underline{$i=0$, $k=1$.}  There is a unique function $h\colon
V(\sSG_{n,1})\to V(SG_{2,1})$ with
$h(S)\intersect\set{0,1,2,3}=S\intersect\set{0,1,2,3}$ for all~$S$.
It satisfies $h(S)\subset S$ and is therefore a graph homomorphism.
Also, $c_{n,1}=c_{2,1}\cmps h$.  It therefore suffices to show that
$c_{2,1}\homot (0\;1\;2)\cmps c_{2,1}$, where $c_{2,1}$ is defined on
$SG_{2,1}$.  We note that $SG_{2,1}$ is a cycle of length~$5$. 
Homomorphisms between cycles have
been studied in more detail by {\v{C}}uki{\'c} and
Koylov~\cite{sonja-dmitry-cycles}.  For our needs the following explicit
construction suffices.
\[
\begin{tabular}{c|c|c|c|c}
$\set{0,3}$&$\set{1,4}$&$\set{0,2}$&$\set{1,3}$&$\set{2,4}$\\
\hline
$0$&$1$&$0$&$1$&$2$\\
$0$&$2$&$0$&$1$&$2$\\
$1$&$2$&$0$&$1$&$2$\\
$1$&$2$&$0$&$1$&$0$\\
$1$&$2$&$0$&$2$&$0$\\
$1$&$2$&$1$&$2$&$0$\\
\end{tabular}
\]
Each row of the table defines a graph homomorphism $SG_{2,1}\to K_3$.
Adjacent rows define homomorphisms which are adjacent in~$[SG_{2,1}, K_3]$.  
\end{proof}

\begin{lem}\label{lem:ds}
Let $n\ge2$, $k\ge1$, $m=2n+k$.  Let $\tau,\rho\in\Aut(SG_{n,k})$ be defined by
\begin{align*}
\tau S&=S+1,
&
\rho S &= k-S,
\intertext{with arithmetic modulo~$m$, and  
$\bar\tau,\bar\rho\in\Aut(K_{k+2})$ by}
\bar\tau x&=x+1,
&
\bar\rho x&=k-x,
\end{align*}
with arithmetic modulo~$k+2$.  Then $c\tau\nbo \bar\tau c$ and
$c\rho\nbo\bar\rho c$, where $c$~is the canonical colouring.
\end{lem}

\begin{proof}
Let $S, S'\in V(SG_{n,k})$.  

If $S\nbo S'$, then $S$ and $S'$ cannot
both contain $m-1$.  Assume $m-1\notin S$.  Then $c\tau(S')\nbo
c\tau(S)=\bar\tau c(S)$ and $c\tau(S)=\bar\tau c(S)\nbo\bar\tau
c(S')$.  This shows $c\tau\nbo\bar\tau c$.

Assume $c\rho(S)\not\nbo\bar\rho c(S')$, i.e.\ $c\rho(S)=i=\bar\rho
c(S')$ for some $0\le i<k+2$.  If $i=k+1$ then
$S=S'=\set{k+1,k+3,\dots,k+2n-1}$.  If $0\le i<k$, then $k-i\in
S\intersect S'$.  In both cases $S\not\nbo S'$.  This shows
$c\rho\nbo\bar\rho c$.
\end{proof}

\begin{proof}[Proof of \prettyref{thm:sg}]
By a theorem of Braun~\cite{braun-sg}, every automorphism of
$SG_{n,k}$ is induced by a permutation of the base set which preserves
its cyclic adjacency relation.  Therefore the elements $\tau$ and
$\rho$ of \prettyref{lem:ds} generate $\Aut(SG_{n,k})$, and it 
suffices to show that $c\homot c\tau$ and $c\homot
c\rho$.  By \prettyref{lem:ds} this reduces to
$c\homot\bar\tau c$ and $c\homot\bar\rho c$.
By \prettyref{prop:sg} these will be true if $\bar\tau$ and $\bar\rho$
are even permutations.  Now $\sign \bar\tau=(-1)^{k+1}=1$ and
$\sign\bar\rho = (-1)^{\frac{k(k+1)}2}=1$.
\end{proof}


\end{document}